\documentclass[12pt]{article}
\usepackage[utf8]{inputenc}\usepackage[english]{babel}
\usepackage[T1]{fontenc}\usepackage{amsmath}
\usepackage{amssymb}\usepackage{amsthm}
\usepackage{amsfonts}\usepackage{mathrsfs}
\usepackage{geometry}\usepackage{enumerate}

\DeclareMathOperator{\convw}{\xrightarrow[]{w}}

\newtheorem{theorem}{Theorem}
\newtheorem{definition}{Definition}
\newtheorem{example}{Example}
\newtheorem{proposition}{Proposition}
\newtheorem{lemma}{Lemma}
\makeatletter\renewcommand{\subsection}{\@startsection{subsection}{1}
{0pt}{3.25ex plus 1ex minus.2ex}{-1em}{\normalfont\normalsize\bf}}\makeatother

\begin{document}

\title{On collectively almost (limitedly, order) L-weakly compact sets of operators}
\author{Safak Alpay, Eduard Emelyanov, Svetlana Gorokhova}
\maketitle

\begin{abstract}
We prove collective versions of semi-duality theorems for sets of
almost (limitedly, and order) L-weakly compact operators. 
\end{abstract}



\section{Introduction and Preliminaries}

\hspace{5mm}
The theory of L-weakly compact operators goes back to P. Meyer-Nieberg.
The most powerful result of the Meyer-Nieberg theory is the duality 
theorem \cite[Satz~3]{Mey1974}. This theory was contributed recently in several papers
(see, for example \cite{AEG2024TJM,BLM2018,CCJ2014,LBM2021,OM2022}).
The aim of the present paper is to prove collective versions of the semi-duality 
theorems \cite[Thm.\,2.5]{BLM2018}, \cite[Thm.\,4]{AEG2024TJM}, and \cite[Thm.\,2.3]{LBM2021}
for almost (limitedly, order) L-weakly compact operators.
Note that the considerably easier duality theorem for collectively L-weakly compact sets
of operators was proved in \cite[Thm.\,2]{E2024}. We also investigate 
the domination problem for collectively almost (limitedly, and order) 
L-weakly compact sets of operators.

Throughout the paper, vector spaces are real; operators are linear and bounded; 
letters $X$ and $Y$ stand for Banach spaces; $E$ and $F$ for Banach lattices; 
$B_X$ denotes the closed unit ball of $X$; $I_X$ the identity operator on $X$;
$E^a$ the order continuous part of $E$;
$\text{\rm L}(X,Y)$ (resp., $\text{\rm K}(X,Y)$, $\text{\rm W}(X,Y)$) 
the space of bounded (resp., compact, weakly compact) operators from $X$ to $Y$. 
A subset $A$ of $F$ is L-weakly compact \cite[pp.\,146--147]{Mey1974}
(briefly, $A$ is an \text{\rm Lwc} set) if each disjoint 
sequence in the solid hull $\text{\rm sol}(A)=\bigcup\limits_{a\in A}[-|a|,|a|]$
of $A$ is norm-null, and an operator $T\in\text{\rm L}(X,F)$ is 
L-weakly compact (briefly, $T\in\text{\rm Lwc}(X,F)$) 
if $T(B_X)$ is an \text{\rm Lwc} set.
We shall use the following lemma (see \cite[Lm.\,1]{AEG2024TJM}).

\begin{lemma}\label{just lemma}
Let $A$ be a nonempty bounded subset of $F'$. The following are equivalent.
\begin{enumerate}[$i)$]
\item $A$ is an \text{\rm Lwc} subset of $F'$.
\item Each disjoint sequence in $B_F$  is uniformly null on $A$.
\item Each disjoint sequence in $B_{F''}$ is uniformly null on $A$.
\end{enumerate}
\end{lemma}

The uniform convergence of a sequence $(g_n)$ of functions  
to a function $g$ on a set $A$ is denoted by $g_n\rightrightarrows g(A)$.
A subset $A$ of $Y$ is limited whenever $f_n\rightrightarrows 0(A)$ 
for every \text{\rm w}$^\ast$-null sequence $(f_n)$ in $Y'$ \cite{BD1984}. 

Following P. M. Anselone and T. W. Palmer \cite{AP1968}, 
a subset ${\cal T}$ of $\text{\rm L}(X,Y)$ is called 
collectively (weakly) compact (collectively \text{\rm Lwc}) 
if the set ${\cal T}B_X=\bigcup\limits_{T\in{\cal T}}T(B_X)$ 
is relatively (weakly) compact (resp., \text{\rm Lwc}). 
Bold letters are used to denote classes of sets of operators
possessing a certain property collectively.
For example, $\text{\bf P}(X,Y)$ denotes collectively $P$-operators
from $X$ to $Y$. In all cases considered in the present paper,
${\cal T}\in\text{\bf P}(X,Y)$ implies $T\in\text{\rm P}(X,Y)$ for every $T\in\cal{T}$. 

\medskip
For unexplained notation and terminology, we refer to 
\cite{AB2006,A1971,Mey1991}.

\section{Collectively qualified L-weakly compact sets of operators}

\hspace{5mm}
We begin with the following definition of collectively qualified L-weakly compact
sets of operators. 

\begin{definition}\label{col Lwc-sets of operators}
{\em
A bounded subset ${\cal T}$ of $\text{\rm L}(X,F)$ is:
\begin{enumerate}[a)]
\item 
collectively L-weakly compact (briefly, ${\cal T}\in\text{\bf Lwc}(X,F)$)
if the set ${\cal T}A=\bigcup\limits_{T\in{\cal T}}T(A)$  
is \text{\rm Lwc} for each bounded $A\subseteq X$.
\item 
collectively almost L-weakly compact (briefly, \text{\bf a-Lwc}(X,F))
if ${\cal T}A$ is \text{\rm Lwc} for each relatively weakly compact $A\subseteq X$.
\item 
collectively limitedly L-weakly compact (briefly, ${\cal T}\in\text{\bf l-Lwc}(X,F)$)
if ${\cal T}A$ is \text{\rm Lwc} for each limited $A\subseteq X$.
\end{enumerate}
A bounded subset ${\cal T}$ of $\text{\rm L}(E,F)$ is:
\begin{enumerate}[d)]
\item 
collectively order L-weakly compact (briefly, ${\cal T}\in\text{\bf o-Lwc}(E,F)$)
if the set ${\cal T}A$ is \text{\rm Lwc} for each order bounded $A\subseteq E$.
\end{enumerate}
}
\end{definition}

\noindent
The next two elementary propositions concern the topological
properties of the sets of operators introduced in Definition \ref{col Lwc-sets of operators}.
The strong closure $\overline{{\cal T}}^s$ of ${\cal T}\in\text{\bf Lwc}(X,F)$ 
also satisfies $\overline{{\cal T}}^s\in\text{\bf Lwc}(X,F)$ \cite[Prop.\,1]{E2024}.

\begin{proposition}\label{closure-LW-sets}
The following holds.
\begin{enumerate}[$i)$]
\item 
${\cal T}\in\text{\bf a-Lwc}(X,F)\Longrightarrow\overline{{\cal T}}^s\in\text{\bf a-Lwc}(X,F)$.
\item 
${\cal T}\in\text{\bf l-Lwc}(X,F)\Longrightarrow\overline{{\cal T}}^s\in\text{\bf l-Lwc}(X,F)$.
\item 
${\cal T}\in\text{\bf o-Lwc}(E,F)\Longrightarrow\overline{{\cal T}}^s\in\text{\bf o-Lwc}(E,F)$.
\end{enumerate}
\end{proposition}

\begin{proof} 
We give a proof of $i)$ only. Other two cases are similar. 
Let ${\cal T}\in\text{\bf a-Lwc}(X,F)$ and $A$ be a weakly compact subset of $X$.
Since ${\cal T}A$ is an \text{\rm Lwc} subset of $F$ then
the norm closure $\overline{{\cal T}A}$ of ${\cal T}A$ is likewise \text{\rm Lwc}
(see, for example, \cite[Lm.\,3]{E2024}). The rest of the proof follows from
$\overline{{\cal T}}^sA\subseteq\overline{{\cal T}A}$.
\end{proof} 

\noindent
Similarly to \cite[Thm.\,1]{E2024}, we have the following result.

\begin{proposition}\label{totally bounded}
{\em
Every totally bounded set of almost (limitedly, order) L-weakly compact operators
is collectively almost (limitedly, order) L-weakly compact.}
\end{proposition}

\begin{proof}
As all three cases are analogues, we include only the proof for \text{\rm a-Lwc} operators.
Let ${\cal T}$ be totally bounded subset of $\text{\rm a-Lwc}(X,F)$ and let $A$ be a 
weakly compact subset of $X$.  For each $k\in\mathbb{N}$ pick a finite subset 
${\cal T}_k$ of ${\cal T}$ such that, for every $T\in{\cal T}$ there exists $S\in{\cal T}_k$ 
with $\|T-S\|\le\frac{1}{2k}$. For each $k\in\mathbb{N}$ the set ${\cal T}_kA$ is \text{\rm Lwc}
by \cite[Cor.\,3.6.4]{Mey1991}. So, ${\cal T}_kA\subseteq[-x_k,x_k]+\frac{1}{2k}B_F$
for some $x_k\in F^a_+$ by \cite[Prop.\,3.6.2]{Mey1991}, and hence
${\cal T}A\subseteq[-x_k,x_k]+\frac{1}{k}B_F$. One more application 
of \cite[Prop.\,3.6.2]{Mey1991} gives that ${\cal T}A$ is an \text{\rm Lwc} set.
It follows ${\cal T}\in\text{\bf a-Lwc}(X,F)$.
\end{proof} 

\subsection{}
The next result strengthens \cite[Thm.\,2]{AEG2024TJM}. Indeed, 
\cite[Thm.\,2]{AEG2024TJM} follows from Theorem \ref{2-LW-sets} by taking ${\cal T}=\{T\}$ and
applying the standard fact that each \text{\rm Lwc} subset of $F$ lies in $F^a$.

\begin{theorem}\label{2-LW-sets}
Let ${\cal T}$ be a bounded subset of $\text{\rm L}(X,F)$. The following are equivalent.
\begin{enumerate}[$i)$]
\item 
${\cal T}\in\text{\bf l-Lwc}(X,F)$.
\item 
The set ${\cal T}x=\{Tx: T\in{\cal T}\}$  
is \text{\rm Lwc} for each $x\in X$.
\item 
$\sup\limits_{T\in{\cal T}}|f_n(Tx)|\to 0$ for every $x\in X$ and 
every disjoint $(f_n)$ in $B_{F'}$.
\end{enumerate}
\end{theorem}

\begin{proof} 
$i)\ \Longrightarrow\ ii)$\
It is trivial.

\medskip
$ii)\ \Longrightarrow\ iii)$\
Let $x\in X$ and let $(f_n)$ be a disjoint sequence in $B_{F'}$.
By the assumption, every disjoint sequence in the bounded set
$\text{\rm sol}({\cal T}x)$ is norm-null, and hence is uniformly
null on $B_{F'}$. By the Burkinshaw--Dodds theorem (see, for example, \cite[Thm.\,5.63]{AB2006}),
$f_n\rightrightarrows 0({\cal T}x)$. Therefore, 
$\sup\limits_{T\in{\cal T}}|f_n(Tx)|\to 0$.

\medskip
$iii)\ \Longrightarrow\ i)$\
Suppose in contrary ${\cal T}A$ is not \text{\rm Lwc} for some limited $A\subseteq X$.
By using \cite[Thm.\,5.63]{AB2006}, take a disjoint sequence $(f_n)$ in $B_{F'}$
such that $f_n\not\rightrightarrows 0({\cal T}x)$. Thus, for some $\varepsilon>0$
there exist  sequences $(x_n)$ in $A$ and $(T_n)$ in ${\cal T}$
satisfying $|(T'_nf_n)x_n|=|f_n(T_nx_n)|\ge\varepsilon$ for all $n$.
Since $A$ is limited, $(T'_nf_n)$ is not \text{\rm w}$^\ast$-null, and hence
$(T'_nf_n)x\not\to 0$ for some $x\in X$. A contradiction to the assumption $iii)$
completes the proof.
\end{proof}

\noindent
The first inclusions in (1) and (2) below are trivial and 
the second ones follow from Theorem \ref{2-LW-sets}.
$$
   \text{\bf Lwc}(X,F)\subseteq\text{\bf a-Lwc}(X,F)\subseteq\text{\bf l-Lwc}(X,F).
   \eqno(1)
$$
$$
   \text{\bf Lwc}(X,F)\subseteq\text{\bf o-Lwc}(X,F)\subseteq\text{\bf l-Lwc}(X,F).
    \eqno(2)
$$
All inclusions in (1) and (2) are proper. It can be seen by taking ${\cal T}=\{T\}$ and considering 
standard examples from \cite{AEG2024TJM,BLM2018,LBM2021,OM2022}.

By the following example (cf. \cite[Ex.\,2]{E2024}), it is quite often that even a bounded set of 
almost (limitedly, order) \text{\rm Lwc} operators need not to be 
collectively almost (limitedly, order) \text{\rm Lwc}.

\begin{example}\label{example}
{\em 
Let $\dim(F^a)=\infty$. Take a positive disjoint normalized sequence $(u_n)$ in $F^a$
and define rank one operators $T_n:\ell^\infty\to F$ by $T_nx=x_nu_n$. 
Then each $T_n$ is \text{\rm Lwc} since $T_n(\ell^\infty)\subseteq F^a$. 
Moreover, $\|T_n\|=1$ for all $n$.
However, ${\cal T}=\{T_n\}_{n=1}^{\infty}\notin\text{\bf l-Lwc}(\ell^\infty,F)$,
and hence ${\cal T}\notin\text{\bf a-Lwc}(\ell^\infty,F)$ and ${\cal T}\notin\text{\bf o-Lwc}(\ell^\infty,F)$.
Indeed, the set $A=\{e_n\}_{n=1}^{\infty}\subset\ell^\infty$ is
limited by Phillip’s lemma (cf. \cite[Thm.\,4.67]{AB2006}), yet
the set ${\cal T}A=\{0\}\cup\{T_ne_n\}_{n=1}^{\infty}=\{0\}\cup\{u_n\}_{n=1}^{\infty}$ 
is not \text{\rm Lwc}.}
\end{example}

\subsection{}
The following theorem is an extension of \cite[Thm.\,2.1]{BLM2018}.  
Theorem \ref{1-LW-sets} implies \cite[Thm.\,2.1]{BLM2018} by taking ${\cal T}=\{T\}$.

\begin{theorem}\label{1-LW-sets}
Let ${\cal T}$ be a bounded subset of $\text{\rm L}(X,F)$. The following are equivalent.
\begin{enumerate}[$i)$]
\item 
${\cal T}\in\text{\bf a-Lwc}(X,F)$.
\item 
The set $\{Tx_k: T\in{\cal T}, k\in\mathbb{N}\}$  
is \text{\rm Lwc} for every \text{\rm w}-convergent $(x_n)$ in $X$.
\item 
$\sup\limits_{T\in{\cal T}}|f_n(Tx_n)|\to 0$
for every disjoint $(f_n)$ in $B_{F'}$ and every \text{\rm w}-convergent $(x_n)$.
\end{enumerate}
\end{theorem}

\begin{proof} 
$i)\ \Longrightarrow\ ii)$\
It is trivial.

\medskip
$ii)\ \Longrightarrow\ iii)$\
Let $x_n\convw x\in X$ and $(f_n)$ be disjoint in $B_{F'}$.
By the assumption, each disjoint sequence in
$\text{\rm sol}({\cal T}(\{x_k\}_{k=1}^\infty))$ is norm-null, and hence is uniformly
null on $B_{F'}$. By \cite[Thm.\,5.63]{AB2006},
$f_n\rightrightarrows 0({\cal T}(\{x_k\}_{k=1}^\infty))$. 
In particular, $\sup\limits_{T\in{\cal T}}|f_n(Tx_n)|\to 0$.

\medskip
$iii)\ \Longrightarrow\ i)$\
Suppose in contrary ${\cal T}A$ is not \text{\rm Lwc} 
for some relatively weakly compact $A\subseteq X$.
By \cite[Thm.\,5.63]{AB2006}, there exists a disjoint sequence $(f_n)$ in $B_{F'}$
such that $f_n\not\rightrightarrows 0({\cal T}A)$. Thus, for some $\varepsilon>0$
there exist sequences $(x_n)$ in $A$ and $(T_n)$ in ${\cal T}$
satisfying $|f_n(T_nx_n)|\ge\varepsilon$ for all $n$.
By using the relative weak compactness of $A$ and passing to a subsequence,
we may assume $x_n\convw x$. A contradiction to the assumption $iii)$
completes the proof.
\end{proof}

\subsection{}
The last theorem of this section extends \cite[Thm.\,2.1]{LBM2021}.  
Theorem \ref{3-LW-sets} implies \cite[Thm.\,2.1]{LBM2021} via taking ${\cal T}=\{T\}$.

\begin{theorem}\label{3-LW-sets}
Let ${\cal T}$ be a bounded subset of $\text{\rm L}(E,F)$. The following are equivalent.
\begin{enumerate}[$i)$]
\item 
${\cal T}\in\text{\bf o-Lwc}(E,F)$.
\item 
The set ${\cal T}(\{x_k\}_{k=1}^\infty)$  
is \text{\rm Lwc} for every order bounded $(x_k)$ in $E$.
\item 
$\sup\limits_{T\in{\cal T}}|f_n(Tx_n)|\to 0$ for every disjoint $(f_n)$ in $B_{F'}$
and every order bounded $(x_n)$.
\end{enumerate}
\end{theorem}

\begin{proof} 
$i)\ \Longrightarrow\ ii)$\
It is trivial.

\medskip
$ii)\ \Longrightarrow\ iii)$\
Let $(x_k)$ be a sequence in $[-x,x]$ for some $x\in X$
and $(f_n)$ be disjoint in $B_{F'}$.
By the assumption, every disjoint sequence in
$\text{\rm sol}({\cal T}(\{x_k\}_{k=1}^\infty))$ is norm-null, and hence is uniformly
null on $B_{F'}$. By \cite[Thm.\,5.63]{AB2006},
$f_n\rightrightarrows 0({\cal T}(\{x_k\}_{k=1}^\infty))$. 
In particular, $\sup\limits_{T\in{\cal T}}|f_n(Tx_n)|\to 0$.

\medskip
$iii)\ \Longrightarrow\ i)$\
Suppose ${\cal T}[-x,x]$ is not \text{\rm Lwc} for some $x\in E_+$.
By \cite[Thm.\,5.63]{AB2006}, there exists disjoint $(f_n)$ in $B_{F'}$
such that $f_n\not\rightrightarrows 0({\cal T}[-x,x])$. Thus, for some $\varepsilon>0$
there exist a sequence $(x_n)$ in $[-x,x]$ and a sequence $(T_n)$ in ${\cal T}$
satisfying $|f_n(T_nx_n)|\ge\varepsilon$ for all $n$.
A contradiction to the assumption $iii)$ completes the proof.
\end{proof}

\section{Main results}

\hspace{5mm}
Recall that an operator $T:E\to Y$ is M-weakly compact if
$(Tx_n)$ is norm null for each disjoint sequence $(x_n)$ in $B_E$.
By the Meyer-Nieberg result \cite[Satz~3]{Mey1974}, 
\text{\rm Lwc} and \text{\rm Mwc} operators are in the duality
(cf. \cite[Prop.\,3.6.11]{Mey1991}, \cite[Thm.\,5.64]{AB2006}).
Almost (limitedly, order) L-weakly compact operators do not enjoy such a duality,
however, for them still there are useful semi-duality results proved in
\cite[Thm.\,2.5]{BLM2018}, \cite[Thm.\,4]{AEG2024TJM}, and \cite[Thm.\,2.3]{LBM2021}
respectively. For establishing the main result of the present paper,
we need the following notion of collectively qualified M-weakly compact
sets of operators.

\begin{definition}\label{col Mwc-sets of operators}
{\em
A bounded subset ${\cal T}$ of $\text{\rm L}(E,Y)$ is:
\begin{enumerate}[a)]
\item 
collectively M-weakly compact (${\cal T}\in\text{\bf Mwc}(E,Y)$) if 
$\sup\limits_{T\in{\cal T}}\|Tx_n\|\to 0$ for every disjoint $(x_n)$ in $B_E$.
\item 
collectively almost M-weakly compact (${\cal T}\in\text{\bf a-Mwc}(E,Y)$) if
$\sup\limits_{T\in{\cal T}}|f_n(Tx_n)|\to 0$ for every weakly convergent $(f_n)$ in $Y'$ and 
every disjoint $(x_n)$ in~$B_E$.
\item 
collectively limitedly M-weakly compact (${\cal T}\in\text{\bf l-Mwc}(E,Y)$) if
$\sup\limits_{T\in{\cal T}}|f(Tx_n)|\to 0$ for every $f\in Y'$ and 
every disjoint $(x_n)$ in~$B_E$.
\end{enumerate}
A bounded subset ${\cal T}$ of $\text{\rm L}(E,F)$ is:
\begin{enumerate}[d)]
\item 
collectively order M-weakly compact (${\cal T}\in\text{\bf o-Mwc}(E,F)$) if
$\sup\limits_{T\in{\cal T}}|f_n(Tx_n)|\to 0$ for every order bounded $(f_n)$ in $F'$ and 
every disjoint $(x_n)$ in~$B_E$.
\end{enumerate}}
\end{definition}

\subsection{The semi-duality theorem.}
Now, we are in the position to present our main result, 
Theorem \ref{CaloLW-MW-duality}. It is worth noting that \cite[Thm.\,2.5]{BLM2018}, 
\cite[Thm.\,4]{AEG2024TJM}, and \cite[Thm.\,2.3]{LBM2021}
follow from Theorem \ref{CaloLW-MW-duality} by taking ${\cal T}=\{T\}$.
It can be also seen by considering examples from \cite{AEG2024TJM,BLM2018,LBM2021,OM2022}
for the case ${\cal T}=\{T\}$ that implications in $(ii_1)$, $(ii_2)$, and $(ii_3)$ are not invertable.

\begin{theorem}\label{CaloLW-MW-duality}
{\em 
The following statements hold:
\begin{enumerate}
\item[$(i)$] 
${\cal T}\in\text{\bf Mwc}(E,Y)\Longleftrightarrow{\cal T}'\in\text{\bf Lwc}(Y',E')$.
\item[$(ii)$] 
${\cal T}\in\text{\bf Lwc}(X,F)\Longleftrightarrow{\cal T}'\in\text{\bf Mwc}(F',X')$.
\end{enumerate}
\begin{enumerate}
\item[$(i_1)$] 
${\cal T}\in\text{\bf a-Mwc}(E,Y)\Longleftrightarrow{\cal T}'\in\text{\bf a-Lwc}(Y',E')$.
\item[$(ii_1)$] 
${\cal T}\in\text{\bf a-Lwc}(X,F)\Longleftarrow{\cal T}'\in\text{\bf a-Mwc}(F',X')$.
\end{enumerate}
\begin{enumerate}
\item[$(i_2)$] 
${\cal T}\in\text{\bf l-Mwc}(E,Y)\Longleftrightarrow{\cal T}'\in\text{\bf l-Lwc}(Y',E')$.
\item[$(ii_2)$] 
${\cal T}\in\text{\bf l-Lwc}(X,F)\Longleftarrow{\cal T}'\in\text{\bf l-Mwc}(F',X')$.
\end{enumerate}
\begin{enumerate}
\item[$(i_3)$]  
${\cal T}\in\text{\bf o-Mwc}(E,F)\Longleftrightarrow{\cal T}'\in\text{\bf o-Lwc}(F',E')$.
\item[($ii_3$)]  
${\cal T}\in\text{\bf o-Lwc}(E,F)\Longleftarrow{\cal T}'\in\text{\bf o-Mwc}(F',E')$.
\end{enumerate}}
\end{theorem}

\begin{proof}
Items $(i)$ and $(ii)$ are proved in \cite[Thm.\,2]{E2024}

\bigskip
$(i_1)$\ 
$(\Longrightarrow)$:\
Let ${\cal T}\in\text{\bf a-Mwc}(E,Y)$ and $f_k\convw f$ in $Y'$. Then 
$\sup\limits_{T\in{\cal T}}|f_{k_n}(Tx_n)|\to 0$ for every 
subsequence $(f_{k_n})$ and every disjoint sequence $(x_n)$ in $B_E$.
Therefore, for each disjoint $(x_n)$ in $B_E$,
$x_n\rightrightarrows 0\big(\{T'f_k\}_{T\in{\cal T}, k\in\mathbb{N}}\big)$. 
By Lemma \ref{just lemma}, $\{T'f_k\}_{T\in{\cal T}, k\in\mathbb{N}}$ is 
an \text{\rm Lwc} subset of $E'$. Since $f_k\convw f$ is arbitrary 
then ${\cal T}'\in\text{\bf a-Lwc}(Y',E')$ by Theorem \ref{1-LW-sets}.

\medskip
$(\Longleftarrow)$:\
Let ${\cal T}'\in\text{\bf a-Lwc}(Y',E')$ and $f_k\convw f$ in $Y'$.
Then $\{T'f_k\}_{T\in{\cal T}, k\in\mathbb{N}}$ is an \text{\rm Lwc} subset of $E'$
by Theorem \ref{1-LW-sets}. Thus, $x_n\rightrightarrows 0(\{T'f_k\}_{T\in{\cal T}, k\in\mathbb{N}})$
by Lemma \ref{just lemma} for each disjoint $(x_n)$ in $B_E$.
In particular, $\sup\limits_{T\in{\cal T}}|f_n(Tx_n)|\to 0$
for each disjoint $(x_n)$ in $B_E$.
Since $f_k\convw f$ is arbitrary then ${\cal T}\in\text{\bf a-Mwc}(E,Y)$.

\medskip
$(ii_1)$\  
Let ${\cal T}'\in\text{\bf a-Mwc}(F',X')$ and $x_n\convw x$ in $X$. 
Then ${\hat x_n}\convw{\hat x}$ in $X''$, and hence
$$
   \sup\limits_{T\in{\cal T}}|f_n(Tx_n)|=\sup\limits_{T\in{\cal T}}|(T'f_n)x_n|=
   \sup\limits_{T\in{\cal T}}|{\hat x_n}(T'f_n)|\to 0
$$
for every disjoint $(f_n)$ in~$B_{E'}$. Since $x_n\convw x$ is arbitrary,
${\cal T}\in\text{\bf a-Lwc}(X,F)$ by Theorem \ref{1-LW-sets}.

\bigskip
$(i_2)$\ 
$(\Longrightarrow)$:\
Let ${\cal T}\in\text{\bf l-Mwc}(E,Y)$ and $f\in Y'$. Then $\sup\limits_{T\in{\cal T}}|f(Tx_n)|\to 0$, 
and hence $x_n\rightrightarrows 0(\{T'f: T\in{\cal T}\})$ for each disjoint $(x_n)$ in $B_E$.
By Lemma \ref{just lemma}, $\{T'f: T\in{\cal T}\}$ is an \text{\rm Lwc} subset of $E'$.
Since $f\in Y'$ is arbitrary, Theorem \ref{2-LW-sets} implies ${\cal T}'\in\text{\bf l-Lwc}(Y',E')$.

\medskip
$(\Longleftarrow)$:\
Let ${\cal T}'\in\text{\bf l-Lwc}(Y',E')$ and $f\in Y'$. Then $\{T'f: T\in{\cal T}\}$ is an \text{\rm Lwc} 
subset of $E'$. By Lemma \ref{just lemma}, $x_n\rightrightarrows 0(\{T'f: T\in{\cal T}\})$,
or equivalently, $\sup\limits_{T\in{\cal T}}|f(Tx_n)|\to 0$ for each disjoint $(x_n)$ in $B_E$.
Since $f\in Y'$ is arbitrary, ${\cal T}\in\text{\bf l-Mwc}(E,Y)$.

\medskip
$(ii_2)$\  
Let ${\cal T}'\in\text{\bf l-Mwc}(F',X')$ and $(f_n)$ be disjoint in $B_{F'}$.
Then $\sup\limits_{T\in{\cal T}}|g(T'f_n)|\to 0$ for each $g\in X''$, and hence 
$\sup\limits_{T\in{\cal T}}|(T'f_n)x|\to 0$ for each $x\in X$.
By Theorem \ref{2-LW-sets}, ${\cal T}\in\text{\bf l-Lwc}(X,F)$.

\bigskip
$(i_3)$\ 
$(\Longrightarrow)$:\
Let ${\cal T}\in\text{\bf o-Mwc}(E,F)$ and $(f_k)$ be order bounded in $F'$. 
Then $\sup\limits_{T\in{\cal T}}|f_{k_n}(Tx_n)|\to 0$ for every subsequence 
$(f_{k_n})$ of $(f_k)$ and every disjoint $(x_n)$ in $B_E$. Therefore, 
$x_n\rightrightarrows 0(\{T'f_k: T\in{\cal T}, k\in\mathbb{N}\})$ for each 
disjoint $(x_n)$ in $B_E$, and hence $\{T'f_k: T\in{\cal T}, k\in\mathbb{N}\}$ is 
an \text{\rm Lwc} subset of $E'$ by Lemma \ref{just lemma}. Since the order bounded 
sequence $(f_k)$ is arbitrary, ${\cal T}'\in\text{\bf o-Lwc}(F',E')$ by Theorem \ref{3-LW-sets}.

\medskip
$(\Longleftarrow)$:\
Let ${\cal T}'\in\text{\bf o-Lwc}(F',E')$ and $(f_k)$ be order bounded in $F'$.
Then, the set $\{T'f_k\}_{T\in{\cal T}, k\in\mathbb{N}}$ is \text{\rm Lwc} in $E'$
by Theorem \ref{3-LW-sets}. Therefore, for each disjoint $(x_n)$ in $B_E$,
$x_n\rightrightarrows 0(\{T'f_k\}_{T\in{\cal T}, k\in\mathbb{N}})$ by Lemma \ref{just lemma}. 
In particular, $\sup\limits_{T\in{\cal T}}|f_n(Tx_n)|\to 0$ for each disjoint $(x_n)$ in $B_E$. 
As the order bounded $(f_k)$ is arbitrary, ${\cal T}\in\text{\bf o-Mwc}(E,F)$.

\medskip
$(ii_3)$\  
Let ${\cal T}'\in\text{\bf o-Mwc}(F',E')$ and $(x_k)$ be order bounded in $X$. 
Then $({\hat x_k})$ is order bounded in $X''$, and hence
$
  \sup\limits_{T\in{\cal T}}|f_n(Tx_n)|=\sup\limits_{T\in{\cal T}}|(T'f_n)x_n|=
  \sup\limits_{T\in{\cal T}}|{\hat x_n}(T'f_n)|\to 0
$
for every disjoint $(f_n)$ in~$B_{E'}$. Since the order bounded sequence $(x_n)$ is arbitrary
then ${\cal T}\in\text{\bf o-Lwc}(E,F)$ by Theorem \ref{3-LW-sets}.
\end{proof}

\subsection{}
Using Theorem \ref{CaloLW-MW-duality}, we obtain the following two propositions.

\begin{proposition}\label{closure-MW-sets}
The following holds.
\begin{enumerate}[$i)$]
\item 
${\cal T}\in\text{\bf Mwc}(E,Y)\Longrightarrow\overline{{\cal T}}^{\|\cdot\|}\in\text{\bf Mwc}(E,Y)$.
\item 
${\cal T}\in\text{\bf a-Mwc}(E,Y)\Longrightarrow\overline{{\cal T}}^{\|\cdot\|}\in\text{\bf a-Mwc}(E,Y)$.
\item 
${\cal T}\in\text{\bf l-Mwc}(E,Y)\Longrightarrow\overline{{\cal T}}^{\|\cdot\|}\in\text{\bf l-Mwc}(E,Y)$.
\item 
${\cal T}\in\text{\bf o-Mwc}(E,F)\Longrightarrow\overline{{\cal T}}^{\|\cdot\|}\in\text{\bf o-Mwc}(E,F)$.
\end{enumerate}
\end{proposition}

\begin{proof} 
As all four cases are analogues, we include only the proof of $ii)$. 
So, let ${\cal T}\in\text{\bf a-Mwc}(E,Y)$.
Theorem \ref{CaloLW-MW-duality} implies ${\cal T}'\in\text{\bf a-Lwc}(Y',E')$.
By Proposition \ref{closure-LW-sets},
$\overline{{\cal T}'}^s\in\text{\bf a-Lwc}(Y',E')$, and hence
$\Big(\overline{{\cal T}}^{\|\cdot\|}\Big)'=\overline{{\cal T}'}^{\|\cdot\|}\in\text{\bf a-Lwc}(Y',E')$.
The second application of Theorem \ref{CaloLW-MW-duality}
gives $\overline{{\cal T}}^{\|\cdot\|}\in\text{\bf a-Mwc}(E,Y)$ as desired.
\end{proof} 

\begin{proposition}\label{Mwc totally bounded}
{\em
Every totally bounded set of almost (limitedly, order) M-weakly compact operators
is collectively almost (limitedly, order) M-weakly compact.}
\end{proposition}

\begin{proof}
We include only the proof for \text{\rm a-Mwc} operators.
Let ${\cal T}$ be totally bounded subset of $\text{\rm a-Mwc}(E,Y)$.  
Then ${\cal T}'\subseteq\text{\rm a-Lwc}(Y',E')$ by Theorem \ref{CaloLW-MW-duality}.
Since ${\cal T}'$ is totally bounded, Proposition \ref{totally bounded}
implies  ${\cal T}'\in\text{\bf a-Lwc}(Y',E')$. The second application 
of Theorem \ref{CaloLW-MW-duality} gives ${\cal T}\in\text{\bf a-Mwc}(E,Y)$ as desired.
\end{proof}

\subsection{The collective domination problem.}
We need the following notion of domination \cite[Def.\,2]{E2024}.

\begin{definition}\label{collective domination}
{\em
Let ${\cal S}$ and ${\cal T}$ be two subsets of $\text{\rm L}_+(E,F)$.
We say that ${\cal S}$ is dominated by ${\cal T}$ if, for each $S\in{\cal S}$
there exists $T\in{\cal T}$ such that $S\le T$.}
\end{definition}

\noindent
It is straightforward that ${\cal S}\subset\text{\rm L}_+(E,F)$
is bounded whenever ${\cal S}$ is dominated by a bounded ${\cal T}\subset\text{\rm L}_+(E,F)$.
Our last result gives a version of \cite[Thm.\,6]{E2024} for 
collectively almost (limitedly, order) L-weakly compact sets of operators.

\begin{theorem}\label{Cal-aloLW-MW-domin}
{\em
Let ${\cal S}\subset\text{\rm L}_+(E,F)$ be dominated by ${\cal T}\subset\text{\rm L}_+(E,F)$. The following hold.
\begin{enumerate}
\item[$(i_1)$]  
If $E$ is a KB-space, then 
${\cal T}\in\text{\bf a-Lwc}(E,F)\Longrightarrow{\cal S}\in\text{\bf a-Lwc}(E,F)$.
\item[$(i_2)$]  
${\cal T}\in\text{\bf l-Lwc}(E,F)\Longrightarrow{\cal S}\in\text{\bf l-Lwc}(E,F)$.
\item[$(ii_2)$] 
${\cal T}\in\text{\bf l-Mwc}(E,F)\Longrightarrow{\cal S}\in\text{\bf l-Mwc}(E,F)$.
\item[$(i_3)$]  
${\cal T}\in\text{\bf o-Lwc}(E,F)\Longrightarrow{\cal S}\in\text{\bf o-Lwc}(E,F)$.
\item[$(ii_3)$] 
${\cal T}\in\text{\bf o-Mwc}(E,F)\Longrightarrow{\cal S}\in\text{\bf o-Mwc}(E,F)$.
\end{enumerate}}
\end{theorem}

\begin{proof}
$(i_1)$\ \ 
Let ${\cal T}\in\text{\bf a-Lwc}(E,F)$ and $A\subseteq E$ be weakly compact.
Since ${\cal S}$ is dominated by ${\cal T}$, then ${\cal S}$ is bounded and
${\cal S}A\subseteq\bigcup\limits_{T\in{\cal T}}T(\text{\rm sol}(A))$.
Indeed, let $x\in A$ and ${\cal S}\ni S\le T_S\in{\cal T}$. 
Since $Sx\le S|x|\le T_S|x|$ then $Sx\in[-T_S|x|,T_S|x|]\subseteq 
T_S(\text{\rm sol}(A))\subseteq\bigcup\limits_{T\in{\cal T}}T(\text{\rm sol}(A))$.
As $E$ is a KB-space, $\text{\rm sol}(A)$ is relatively weakly compact by the
Abramovich--Wickstead theorem (cf., \cite[Thm.\,4.39]{AB2006}).
Since ${\cal T}\in\text{\bf a-Lwc}(E,F)$ then 
$\bigcup\limits_{T\in{\cal T}}T(\text{\rm sol}(A))$ is an \text{\rm Lwc} set,
and hence ${\cal S}A$ is an \text{\rm Lwc} set. Therefore, ${\cal S}\in\text{\bf a-Lwc}(E,F)$.

\medskip
$(i_2)$\ \ 
Let ${\cal T}\in\text{\bf l-Lwc}(E,F)$ and $x\in E$.
Then $\{T|x|: T\in{\cal T}\}$, and hence $\{[-T|x|,T|x|]: T\in{\cal T}\}$ is an \text{\rm Lwc} set.
Since ${\cal S}$ is dominated by ${\cal T}$,
$$
   \{Sx: S\in{\cal S}\}\subseteq\{[-T_S|x|,T_S|x|]: S\in{\cal S}\}\subseteq 
   \{[-T|x|,T|x|]: T\in{\cal T}\}.
   \eqno(3)
$$
Since $x\in E$ is arbitrary, (3) implies ${\cal S}\in\text{\bf l-Lwc}(E,F)$ by Theorem \ref{2-LW-sets}.

$(ii_2)$\ \ 
Let ${\cal T}\in\text{\bf l-Mwc}(E,F)$, $f\in Y'$ and $(x_n)$ be disjoint in $B_E$.
It follows $\sup\limits_{T\in{\cal T}}|f|(T|x_n|)\to 0$, and hence
$\sup\limits_{S\in{\cal T}}|f(Sx_n)|\to 0$ because ${\cal S}$ is dominated by ${\cal T}$.
Therefore, ${\cal S}\in\text{\bf l-Mwc}(E,F)$.

\medskip
$(i_3)$\ \ 
Let ${\cal T}\in\text{\bf o-Lwc}(E,F)$ and $x\in E_+$.
Since ${\cal S}$ is dominated by ${\cal T}$ then
${\cal S}[-x,x]\subseteq\bigcup\limits_{T\in{\cal T}}T[-x,x]$.
Since ${\cal T}\in\text{\bf o-Lwc}(E,F)$ then 
$\bigcup\limits_{T\in{\cal T}}T[-x,x]$ is an \text{\rm Lwc} set,
and hence ${\cal S}[-x,x]$ is likewise \text{\rm Lwc}. 
Therefore, ${\cal S}\in\text{\bf o-Lwc}(E,F)$.

$(ii_3)$\ \ 
Let ${\cal T}\in\text{\bf o-Mwc}(E,F)$,
$(f_n)$ be order bounded in $F'$, and 
$(x_n)$ be disjoint in $B_E$.
Then $(|f_n|)$ is order bounded in $F'$ and 
$(|x_n|)$ is disjoint in $B_E$, and hence 
$\sup\limits_{T\in{\cal T}}||f_n|(T|x_n|)|\to 0$
because ${\cal T}\in\text{\bf o-Mwc}(E,F)$.
Since ${\cal S}$ is dominated by ${\cal T}$,
for every $S\in{\cal S}$ there exists $T_S\in{\cal T}$
with $|f_n(Sx_n)|\le|f_n|(S|x_n|)\le|f_n|(T_S|x_n|)$.
So, $\sup\limits_{S\in{\cal S}}|f_n(Sx_n)|\to 0$, and hence
${\cal S}\in\text{\bf o-Mwc}(E,F)$.
\end{proof}


\bibliographystyle{plain}

\addcontentsline{toc}{section}{KAYNAKLAR}
\begin{thebibliography}{99}
\normalsize	

\bibitem{AB2006}
Aliprantis, C.D., Burkinshaw, O. 
Positive operators.
Springer, Berlin (2006).

\bibitem{AEG2024TJM}
Alpay, S., Emelyanov, E., Gorokhova, S. 
Duality and norm completeness in the classes of limitedly Lwc and Dunford--Pettis Lwc operators. 
Turk. J. Math. 48(1), 267--278 (2024).

\bibitem{AEG2024Positivity}
Alpay, S., Emelyanov, E., Gorokhova, S. 
Enveloping norms of regularly P-operators in Banach lattices. 
Positivity 28, 37 (2024).

\bibitem{A1971}
Anselone, P.M.
Collectively Compact Operator Approximation Theory and Applications to Integral Equations.
Prentice-Hall (1971).

\bibitem{AP1968}
Anselone, P.M., Palmer, T.W.
Collectively compact sets of linear operators.
Pacific J. Math. 25(3), 417--422 (1968).

\bibitem{BD1984}
Bourgain, J., Diestel, J.
Limited operators and strict cosingularity.
Math. Nachr. 119, 55--58 (1984).

\bibitem{BLM2018}
Bouras, K., Lhaimer, D., Moussa, M.
On the class of almost L-weakly and almost M-weakly compact operators.
Positivity 22, 1433--1443 (2018).

\bibitem{CCJ2014}
Chen, J.X., Chen, Z.L., Ji, G.X.
Almost limited sets in Banach lattices. 
J. Math. Anal. Appl. 412(1), 547--553 (2014).

\bibitem{E2024}
Emelyanov, E.
On collectively L-weakly compact sets of operators.
arxiv.org/abs/2407.10455v2 (2024).

\bibitem{EG2023}
Emelyanov, E., Gorokhova, S. 
Operators affiliated to Banach lattice properties and their enveloping norms.
Turk. J. Math. 47(6), 1659--1673 (2023).


\bibitem{Kus2000}
Kusraev, A.G.
Dominated Operators.
Kluwer, Dordrecht (2000).

\bibitem{LBM2021}
Lhaimer, D., Bouras, K., Moussa, M.
On the class of order L- and order M-weakly compact operators.
Positivity 25, 1569--1578 (2021).

\bibitem{Mey1974}
Meyer-Nieberg, P.
\"{U}ber klassen schwach kompakter Operatoren in Banachverb\"{a}nden.
Math. Z. 138, 145--159 (1974).

\bibitem{Mey1991}
Meyer-Nieberg, P.
Banach lattices.
Universitext. Springer, Berlin (1991).

\bibitem{OM2022}
Oughajji, F.Z., Moussa, M.
Weak M weakly compact and weak L weakly compact operators.
Afr. Mat. 33, 34 (2022).

\end{thebibliography}
\end{document}